\documentclass{amsart}
\usepackage{amsfonts}
\usepackage{graphicx}
\usepackage{tabularx}

\usepackage[usenames,dvipsnames]{color}
\usepackage{comment}

\newcommand{\kommentar}[1]{}
\newcommand{\C}{\mathbb{C}}

\newcommand{\Q}{\mathbb{Q}}

\newcommand{\Z}{\mathbb{Z}}

\newcommand{\HH}{\mathbb{H}}

\newcommand{\TT}{\mathbb{T}}

\newcommand{\m}{\mathrm{m}}

\newtheorem{thm}{Theorem}
\newtheorem{defn}[thm]{Definition}

\newtheorem{lem}[thm]{Lemma}

\newtheorem{rem}[thm]{Remark}

\newcommand{\Li}{\mathrm{Li}} 
\newcommand{\im}{\mathop{\mathrm{Im}}} 

\title{Regulator proofs for Boyd's identities on genus 2 curves}

\author[M. \ Lal\'in]{Matilde Lal\'in}
\address{Universit\'e de Montr\'eal, Pavillon Andr\'e-Aisenstadt, D\'ept. de math\'ematiques et de statistique, CP 6128, succ. Centre-ville Montréal, Qu\'ebec, H3C 3J7, Canada}
\email{mlalin@dms.umontreal.ca}

\author[G. Wu]{Gang Wu}
\address{Universit\'e de Montr\'eal, Pavillon Andr\'e-Aisenstadt, D\'ept. de math\'ematiques et de statistique, CP 6128, succ. Centre-ville Montréal, Qu\'ebec, H3C 3J7, Canada}
\email{gang.wu.1@umontreal.ca }

\thanks{This research was supported by the Natural Sciences and Engineering Research Council
of Canada [Discovery Grant 355412-2013]}

\subjclass[2010]{Primary 11R06; Secondary  11G05, 11F66, 19F27, 33E05}
\keywords{Mahler measure; special values of $L$-functions; genus 2 curves; elliptic curve; elliptic regulator}

\begin{document}

\maketitle

\begin{abstract}
We use the elliptic regulator to
recover some identities between Mahler measures involving certain families of genus 2 curves
that were conjectured by Boyd and proven by Bertin and Zudilin by differentiating the Mahler measures and using hypergeometric identities. Since our proofs involve the regulator, they yield light 
into the expected relation of each Mahler measure to special values of $L$-functions of certain elliptic curves. 

\end{abstract}

\section{Introduction}

The (logarithmic) Mahler measure of a non-zero rational function $P \in \C(x_1,\dots,x_n)$ is defined by 
\begin{equation*}
 \m(P):=\frac{1}{(2\pi i)^n}\int_{\mathbb{T}^n}\log|P(x_1,\dots, x_n)|\frac{dx_1}{x_1}\cdots \frac{dx_n}{x_n},
\end{equation*}
where $\mathbb{T}^n=\{(x_1,\dots,x_n)\in \mathbb{C}^n : |x_1|=\cdots=|x_n|=1\}$. 

This definition often yields special values of functions with number theoretic significance, 
such as the Riemann zeta-function and $L$-functions associated to arithmetic-geometric objects such as elliptic curves. 

The relationship with $L$-functions of elliptic curves was predicted by Deninger \cite{Deninger} who connected the Mahler measure to certain regulator expected to be related to a special value of an $L$-function by means of Be\u\i linson's conjectures.  
This was further investigated in detail by Boyd \cite{Bo98}, who conducted a systematic study of certain families of two-variable polynomials and found many numerical examples such as 
\begin{equation}\label{eq:identity}
\m(P_\alpha)\stackrel{?}{=}r_\alpha L'(E_\alpha,0),
\end{equation}
where 
\[P_\alpha(x,y)=(x+1)(y+1)(x+y)-\alpha xy,\]
$\alpha$ is an integer, $r_\alpha$ is a rational number of small height, 
and 
\begin{equation}\label{eq:elliptic}
E_\alpha:Y^2+(\alpha-2)XY+\alpha Y=X^3
\end{equation} is the elliptic curve corresponding to the zero loci of $P_\alpha(x,y)$. 

Rodriguez-Villegas \cite{RV} studied these identities in the context of Be\u\i linson's conjectures
and was able to prove those related to elliptic curves of conductor 36, which have complex multiplication. 
Subsequently, other cases were proven involving conductors 14 and 20 (see Table \ref{table}). 
\begin{table}\label{table}
\begin{tabular}{|r|c|c|l|}
\hline
$\alpha$ & $r_\alpha$ & $N$ & Proven by\\
\hline
\hline
$-4$ & $2$ & $36$  & Rodriguez-Villegas \cite{RV}\\
$2$ & $1/2$ & $36$ & Rodriguez-Villegas \cite{RV}\\
$-8$ & $10$ & $14$ & Mellit \cite{Me12}\\
$1$ & $1$ & $14$ & Mellit \cite{Me12}\\
$7$ & $6$ & $14$ & Mellit \cite{Me12}\\
$-2$ & $3$ & $20$ & Rogers--Zudilin \cite{RZ12}\\
$4$ & $2$ & $20$ & Rogers--Zudilin \cite{RZ12}\\
\hline
\end{tabular}
\caption{Proven cases of formula \eqref{eq:identity}.}
\end{table}

Boyd \cite{Bo98} also investigated some families of higher genus such as 
\[Q_\alpha(x,y)=(x^2+x+1)y^2+\alpha x(x+1) y + x(x^2+x+1),\]
\[R_\alpha(x,y)=(x^2+x+1)y^2+(x^4+\alpha x^3+(2\alpha-4)x^2+\alpha x+1)y+x^2(x^2+x+1),\]
and
\[S_\alpha(x,y)=y^2+(x^4+\alpha x^3+2\alpha x^2+\alpha x +1) y +x^4,\]
which correspond to families (3-3), (3-9), and (3-13) in \cite{Bo98}. 

The Jacobian associated to the genus 2 curve $S_\alpha(x,y)=0$ splits as a product of two elliptic curves, one of which is
$E_\alpha$ given by \eqref{eq:elliptic}. Boyd found numerical relations of the type
\[\m(S_\alpha)\stackrel{?}{=}s_\alpha L'(E_\alpha,0).\]
Similarly, the Jacobians associated to $Q_\alpha(x,y)=0$ and $R_\alpha(x,y)=0$ can also be written as product of 
two elliptic curves, with a common factor, and Boyd found that the Mahler measures are numerically related to the $L$-function corresponding to the common factor. 

The above findings led Boyd to conjecture relationships 
between the Mahler measures of $P_\alpha$ and $S_\alpha$ and between the Mahler measures of 
$Q_\alpha$ and $R_\alpha$. These results were eventually proven by Bertin and Zudilin and are summarized as follows.

\begin{thm}\label{thm:BZ}\cite{BZ} For $\alpha$ taking real values, we have
 \[\m(S_\alpha) = \left\{\begin{array}{cl} 2 \m(P_{\alpha})& 0\leq \alpha \leq 4,\\ \\ 
 \m(P_{\alpha}) & \alpha \leq -1. \end{array}\right.\]
\end{thm}

\begin{thm} \label{thm:BZ2} \cite{BZ2} For real $\alpha\geq 4$, we have
\[\m(Q_\alpha)=\m(R_{2+\alpha}).\]
 \end{thm}

Both results are achieved by studying the Mahler measures as functions on the parameter $\alpha$ and by differentiating respect to $\alpha$. The equalities are then established by using several hypergeometric identities. 

The family $S_\alpha(x,y)$ was also studied by Bosman in his thesis \cite{Bosman}. He considered the relationship with the regulator
and proved exact formulas for $\m(S_\alpha)$ in the cases $\alpha=8$, corresponding to genus 0, and 
$\alpha=-1,2$ corresponding to genus 1 ($\alpha=-1$ yields a degenerate case of a Dirichlet $L$-function, while $\alpha=2$ corresponds to an elliptic curve with complex multiplication). 
There is another case of genus 1 corresponding to $\alpha=4$, which can be proven by techniques of 
modular unit parametrizations \cite{Zu14,BZ} and the curve has genus 2 for all the other values of $\alpha$ \cite{Bosman}.  

Bosman \cite{Bosman} uses the regulator to relate the Mahler measure of the family $S_\alpha(x,y)$ to a combination 
of elliptic dilogarithms. A proof of Theorem \ref{thm:BZ} could be achieved by relating those elliptic dilogarithms to the ones 
corresponding to the Mahler measure of $P_\alpha(x,y)$. 

The goal of these notes is to provide further clarification by reproving Theorems \ref{thm:BZ} and 
\ref{thm:BZ2} by using the regulator theory. In the case of 
Theorem \ref{thm:BZ} this finishes the work started by Bosman. In the case of Theorem \ref{thm:BZ2}, the development is entirely new. 
Our proofs show the role of the regulator in these relationships, which is the key step to an eventual understanding of the conjectural relationship of each of these Mahler measures to their corresponding $L$-value.

This paper is organized as follows. Section \ref{sec:regulator} presents a general exposition of the relationship between Mahler measure and elliptic regulators. 
Then sections \ref{sec:1} and \ref{sec:2} treat Theorems \ref{thm:BZ} and \ref{thm:BZ2} respectively.
Each of these sections is separated into a first part where the relationship between the regulators is considered, and a second part, where the relationship between the cycles is considered.

\section{The regulator theory} \label{sec:regulator}

Here we recall the definition of the regulator on the second $K$-group of an elliptic curve $E$ 
given by Bloch  and Be\u\i linson  and explain how it can be computed in terms of the elliptic dilogarithm. 
We then discuss the relationship between Mahler measure and the regulator. 

Let $F$ be a field. Matsumoto's theorem implies that the second $K$-group of $F$ can be described as
\[K_2(F) \cong \Lambda^2 F^\times/\{x\otimes (1-x): x \in F, x\not = 0,1\}.\]

Let $E/\Q$ be an elliptic curve given by an equation $P(x,y)=0$. 
Rodriguez-Villegas \cite{RV} describes certain conditions on the polynomial $P$ that guarantee
the triviality of tame symbols and that $K_2(E)\otimes \Q \subset K_2(\Q(E)) \otimes \Q$.

Let $x,y \in \Q(E)$. We will work with the differential form 
\begin{equation}\label{eq:eta}
\eta(x,y) := \log |x| d \arg y - \log|y| d \arg x,
\end{equation}
where $d \arg x$ is defined by $\im(dx/x)$. 

The Bloch--Wigner dilogarithm is given by 
\begin{equation}\label{Bloch-Wigner}
D(x)= \im(\Li_2(x))+\arg(1-x) \log|x|,
\end{equation}
where
\[\Li_2(x)=-\int_0^x\frac{\log(1-z)}{z}dz.\]

The form $\eta(x,y)$ is closed in its domain of definition,  multiplicative, antisymmetric, and satisfies
\[\eta(x,1-x) = d D(x).\]

\begin{defn}The regulator map of Bloch \cite{Bloch} and Be\u\i linson \cite{Beilinson} is given by
\begin{eqnarray*}
r_E:K_2(E)\otimes \Q &\rightarrow& H^1(E,\mathbb{R})\\
\{x,y\}&\rightarrow &\left\{ [\gamma] \rightarrow \int_\gamma \eta(x,y)\right \}.
 \end{eqnarray*}
\end{defn} 
In the above definition, we take $[\gamma] \in H_1(E,\mathbb{Z})$ and interpret $H^1(E,\mathbb{R})$ as the dual of $H_1(E,\mathbb{Z})$. 


\begin{rem}\label{remark}
Due to the action of complex conjugation on $\eta$, the regulator map is trivial for the classes that remain invariant by complex conjugation, denoted by  $H_1(E,\mathbb{Z})^+$. It therefore suffices to consider the regulator as a function on $H_1(E,\mathbb{Z})^-$, a one-dimensional space.
\end{rem}

For  $E/\Q$ an elliptic curve, we have
\begin{equation}
\label{eq:diagrammaps}\begin{array}{ccccc}
E(\C)& \stackrel{\sim}{\rightarrow} &\C/(\Z+\tau\Z) &\stackrel{\sim}{\rightarrow}&\C^\times /q^\Z\\ \\
P=(\wp(u),\wp'(u))& \rightarrow & u \bmod \Lambda & \rightarrow & z=e^{2\pi i u},
\end{array}
\end{equation}
where $\wp$ is the Weierstrass function, $\Lambda$ is the lattice $\Z+\tau\Z$, $\tau \in \HH$, and $q=e^{2 \pi i \tau}$.

The next definition is due to Bloch \cite{Bloch}.
\begin{defn}
The elliptic dilogarithm is a function on $E(\C)$ given for $P\in E(\C)$ corresponding to $z \in \C^\times/q^\Z$ by 
\begin{equation}\label{eq:elldilogdef}
D^E(P):=\sum_{n \in \Z} D(q^nz),
\end{equation}
where $D$ is the Bloch--Wigner dilogarithm defined by \eqref{Bloch-Wigner}.
\end{defn}

Let $\Z[E(\C)]$ be the group of divisors on $E$ and let
\[\Z[E(\C)]^-\cong \Z[E(\C)]/\{ (P)+(-P): P\in E(\C)\}.\]

Let $x,y \in \C(E)^\times$. We define a diamond operation by 
\begin{eqnarray*}
\diamond: \Lambda^2 \C(E)^\times &\rightarrow &\Z[E(\C)]^-\\
(x)\diamond (y) &=& \sum_{i,j}m_in_j (S_i-T_j),
\end{eqnarray*}
where \[(x)=\sum_i m_i (S_i) \mbox{ and  }(y)=\sum_j n_j (T_j).\]

With these elements, we have the following result.
\begin{thm}\label{thm:bloch}(Bloch \cite{Bloch}) The elliptic dilogarithm $D^E$ extends by linearity to a map from $\Z[E(\Q)]^-$ to $\C$. Let $x, y \in \Q(E)$ and $\{x,y\}\in K_2(E)$. Then
\[r_E(\{x,y\})[\gamma]=D^E((x)\diamond(y)),\]
where $[\gamma]$ is a generator of $H_1(E,\Z)^-$.
 \end{thm}

 Deninger \cite{Deninger} was the first to write a formula of the form
\begin{equation}\label{eq:mregulator}
 \m(P) = \frac{1}{2\pi}r(\{x,y\})[\gamma].
 \end{equation}
Rodriguez-Villegas combined the above expression with Theorem \ref{thm:bloch} to prove an identity between two Mahler measures (originally conjectured in Boyd \cite{Bo98})
in \cite{RV2}. This was the first of the type of result that we consider in these notes. 

Let $P(x,y) \in \C[x,y]$ be a polynomial of degree 2 on $y$. 
We may then write 
\[P(x,y)=P^*(x)(y-y_1(x))(y-y_2(x)),\]
where $y_1(x), y_2(x)$ are algebraic functions. 

We recall a particular case of Jensen's formula. Let $\alpha \in \C$. Then
\[\frac{1}{2\pi i}\int_{\TT^1} \log |z-\alpha| \frac{dz}{z} = \left\{\begin{array}{cc}\log |\alpha| & |\alpha|\geq 1, \\ 0 & |\alpha|\leq 1.\end{array}
\right. \]

By applying Jensen's formula in the Mahler measure formula of $P(x,y)$ with respect to the variable $y$, we have
\begin{align*}
\m(P)-\m(P^*)=&\frac{1}{(2 \pi i)^2} \int_{\TT^2}\log |P(x,y)|\frac{dx}{x}\frac{dy}{y}-\m(P^*)\\
=& \frac{1}{(2 \pi i)^2} \int_{\TT^2}(\log |y-y_1(x)|+\log|y-y_2(x)|)\frac{dx}{x}\frac{dy}{y}\\
=& \frac{1}{2\pi i} \int_{|x|=1,|y_1(x)|\geq 1} \log|y_1(x)|\frac{dx}{x}+\frac{1}{2\pi i} \int_{|x|=1,|y_2(x)|\geq 1} \log|y_2(x)|\frac{dx}{x}.
\end{align*}
Recalling formula \eqref{eq:eta} for $\eta(x,y)$, we have,
\begin{align*}
\m(P)-\m(P^*) =& -\frac{1}{2 \pi} \int_{|x|=1,|y_1(x)|\geq 1} \eta(x,y_1)-\frac{1}{2 \pi} \int_{|x|=1,|y_2(x)|\geq 1} \eta(x,y_2).
\end{align*}

Often we will encounter the case that one of the roots $y_1(x)$ has always absolute value greater than or equal to 1 as $|x|=1$ 
and the other root has always absolute value smaller than or equal to 1 as $|x|=1$. This will allow us to write the right-hand side as a single term,
an integral over a closed path. 

When $P$ corresponds to an elliptic curve and when the set $\{|x|=1,|y_i(x)|\geq 1\}$ can be seen as a cycle in $H_1(E,\Z)^-$, then we may be able to recover a formula of the type \eqref{eq:mregulator}. This has to be examined on a case by case basis.


\section{The families from Theorem \ref{thm:BZ}}\label{sec:1}

\subsection{The relationship between the regulators}

Recall that the family $P_\alpha(x,y)$ is given by 
\[P_\alpha(x,y)=(x+1)(y+1)(x+y)-\alpha xy,\]
which is birational to the Deuring form
\[E_\alpha:Y^2+(\alpha-2)XY+\alpha Y=X^3.\]
The change of variables is given by 
\begin{align*}
 X(x,y)&=\alpha\frac{x+y+1}{x+y-\alpha}, & & x(X,Y)=\frac{X-Y}{X-\alpha},\\
Y(x,y)&=\alpha \frac{-\alpha x +y+1}{x+y-\alpha}, & & y(X,Y)=\frac{Y+(\alpha-1)X+\alpha}{X-\alpha}.
\end{align*}

The torsion group of $E_\alpha$ for $\alpha \in \Q$ has order 6, generated by $P=(\alpha,\alpha)$, with $2P=(0,0)$, $3P=(-1,-1)$, $4P=(0,-\alpha)$, $5P=(\alpha,-\alpha^2)$.

Our first goal is to compute the diamond operation $(x)\diamond (y)$ in $E_\alpha$. This will allow us to understand the differential 
form $\eta(x,y)$ that is involved in the computation of $\m(P_\alpha)$. Thus, we proceed to compute the divisors $(x)$ and $(y)$. 
\begin{align*}
(x)&= ((P)+(2P)+(3P)-3O)-((P)+(5P)-2O) \\
&= (2P)+(3P)-(5P)-O\\
(y)&= ((3P)+(4P)+(5P)-3O)-((P)+(5P)-2O) \\
&=-(P)+(3P)+(4P)-O
\end{align*}
The diamond operation yields
\begin{align}
(x)\diamond (y)&=-6(P)-6(2P)\label{eq:xy}.
\end{align}

Now we proceed to compute the diamond operation $(x_1)\diamond (y_1)$. Recall that 
\[S_\alpha(x_1,y_1)=y_1^2+(x_1^4+\alpha x_1^3+2\alpha x_1^2+\alpha x_1 +1) y_1 +x_1^4.\]
Bosman (\cite{Bosman}, p. 47) considers the curve
\[C_\alpha: Y_1^2=h_1(X_1^2),\]
where
\[h_1(Z_1)=(\alpha^2+\alpha)Z_1^3+(-2\alpha^2+5\alpha+4)Z_1^2+(\alpha^2-5\alpha+8)Z_1-\alpha+4.\]

We have rational maps 
\[\begin{array}{ccccc}
E_\alpha &\stackrel{\phi}{\rightarrow}& C_\alpha &\stackrel{\psi}{\rightarrow} &\{S_\alpha=0\}\\ \\
(Z_1,Y_1)&\rightarrow &(X_1, Y_1)& \rightarrow& (x_1,y_1)
 \end{array}\]
defined by $\phi(Z_1)=X_1^2$, $\phi(Y_1)=Y_1$, while
$\psi^{-1}$ and $\psi$ are given by 
\begin{align*}
X_1(x_1,y_1)=&\frac{x_1+1}{x_1-1},& & x_1(X_1,Y_1)=\frac{X_1+1}{X_1-1},\\ 
Y_1(x_1,y_1)=&\frac{4(y_1^2-x_1^4)}{y_1(x_1-1)^3(x_1+1)},& & y_1(X_1,Y_1)=\frac{2X_1Y_1-(2\alpha +1)X_1^4+(2\alpha-6)X_1^2-1}{(X_1-1)^4},
\end{align*}
respectively.

The relationship between the rational functions $Z_1,Y_1$ and $X,Y$ in $E_\alpha$ is given by the following transformations.
\begin{align*}
Y_1(\alpha^2+\alpha)=&4(2Y+(\alpha-2)X+\alpha),\\ 
(\alpha^2+\alpha)Z_1-(\alpha^2-3\alpha)=&4X,
\end{align*}
so that
\begin{align*}
Y=&\frac{\alpha((\alpha+1)Y_1+(-\alpha^2+\alpha+2)Z_1+(\alpha^2-5\alpha+2))}{8},\\
Y_1=&\frac{4(2Y+(\alpha-2)X+\alpha)}{\alpha^2+\alpha},\\
 Z_1=&\frac{4X+\alpha^2-3\alpha}{\alpha^2+\alpha}.
\end{align*}

Our goal is to compute 
\[r_{C_\alpha}(\{x_1(X_1,Y_1),y_1(X_1,Y_1)\})[\psi\circ \gamma],\]
where $\gamma$ is the path in $\{S_\alpha=0\}$ defined by $|x_1|=1, |y_1|\geq 1$ (for certain choice of a root $y_1$), that will be made precise later. 
In order to do this, we will consider the pushforward by $\phi$ to the regulator $r_{E\alpha}$ in $E_\alpha$. 

Bosman does this by finding rational functions $a(Z_1,Y_1)$, $b(Z_1,Y_1)$ such that 
\[a(X_1^2,Y_1)x_1(X_1,Y_1)+b(X_1^2,Y_1)y_1(X_1,Y_1)=1.\]
Then, he proves the following result, which we reproduce here for completeness. 
\begin{lem}\label{lem:ab} We have
\[r_{C_\alpha}(\{x_1(X_1,Y_1),y_1(X_1,Y_1)\})[\psi\circ \gamma]=
 -r_{E_\alpha}(\{a(Z_1,Y_1),b(Z_1,Y_1)\})[\phi\circ \psi\circ \gamma].
\]

\end{lem}
\begin{proof}
It suffices to see the identity at the level of the diamond operator, namely, to prove that
\[(x_1(X_1,Y_1))\diamond (y_1(X_1,Y_1))\sim -(a(X_1^2,Y_1))\diamond (b(X_1^2,Y_1)).\]

Because of the triviality of the Steimberg symbol, $(f)\diamond(1-f)\sim 0$, and 
\begin{align*}
0 \sim &(a(X_1^2,Y_1)x_1(X_1,Y_1))\diamond (b(X_1^2,Y_1)y_1(X_1,Y_1))\\
\sim &(a(X_1^2,Y_1))\diamond (b(X_1^2,Y_1))+ (x_1(X_1,Y_1))\diamond  (b(X_1^2,Y_1))\\
&+ (a(X_1^2,Y_1)) \diamond (y_1(X_1,Y_1)) + (x_1(X_1,Y_1))\diamond (y_1(X_1,Y_1)).
\end{align*}
Now consider the automorphism of $S_\alpha(x_1,y_1)=0$ given by $x_1\rightarrow \frac{1}{x_1}$, $y_1\rightarrow \frac{1}{y_1}$.
We remark that $X_1\rightarrow -X_1$ and $Y_1\rightarrow Y_1$. Then
\begin{align*}
0 \sim &(a(X_1^2,Y_1)x_1(-X_1,Y_1))\diamond (b(X_1^2,Y_1)y_1(-X_1,Y_1))\\
\sim &(a(X_1^2,Y_1))\diamond (b(X_1^2,Y_1))- (x_1(X_1,Y_1))\diamond  (b(X_1^2,Y_1))\\
&- (a(X_1^2,Y_1)) \diamond (y_1(X_1,Y_1)) + (x_1(X_1,Y_1))\diamond (y_1(X_1,Y_1)).
\end{align*}
Combining the above expressions, we obtain the result. 
\end{proof}

Following Bosman, we take 
\begin{align*}
a(Z_1,Y_1)=&\frac{(-Z_1^2-6Z_1-1)Y_1+(4\alpha+2)Z_1^3+14Z_1^2+(-4\alpha+14)Z_1+2}{(Z_1-1)((-Z_1-1)Y_1+(2\alpha+1)Z_1^2+(-2\alpha+6)Z_1+1)}\\
=&\frac{2X^2Y+4\alpha^2XY+(\alpha^4-2\alpha^3-\alpha^2)Y+(-3\alpha-4)X^3+(-\alpha^3+2\alpha)X^2+(\alpha^3+2\alpha^2)X-\alpha^3}{(X-\alpha)((\alpha^2-\alpha)Y+2XY-(\alpha+3)X^2+2\alpha X)}
\end{align*}
and
\begin{align*}
b(Z_1,Y_1)=&\frac{(Z_1-1)^2}{(-Z_1-1)Y_1+(2\alpha+1)Z_1^2+(-2\alpha+6)Z_1+1}\\
=&-\frac{(X-\alpha)^2}{(\alpha^2-\alpha)Y+2XY-(\alpha+3)X^2+2\alpha X}.
\end{align*}

We proceed to compute the diamond operation for $(a(X,Y))$ and $(b(X,Y))$.
Consider the following points on $E_\alpha$. 
\begin{align*}
P=&(\alpha,\alpha),\\
U_\pm=&\left(\frac{\alpha(-\alpha\pm\sqrt{\alpha^2-16\alpha+32})}{8},\frac{\alpha^2(\alpha-8\mp \sqrt{\alpha^2-16\alpha+32})}{16} \right),\\
V_\pm=&\left(\frac{-\alpha^2+4\alpha-3\pm (\alpha+1)\sqrt{\alpha^2-10\alpha+9}}{8},\frac{\alpha^3-7\alpha^2-\alpha-9\mp (\alpha^2-2\alpha-3)\sqrt{\alpha^2-10\alpha+9}}{16}\right),\\
\end{align*}
where we also have that $U_++U_-=P$ and $V_++V_-=2P$. Thus we write $U$ for $U_+$, $V$ for $V_+$, $P-U$ for $U_-$, 
and $2P-V$ for $V_-$.

One can check that 
\begin{align*}
(X-\alpha) =&(P)+(5P)-2O,\\
((\alpha^2-\alpha)Y+2XY-(\alpha+3)X^2+2\alpha X)=&2(P)+(2P)+(V)+(2P-V)-5O,\\
\end{align*}
and
\begin{align*}
&(2X^2Y+4\alpha^2XY+(\alpha^4-2\alpha^3-\alpha^2)Y+(-3\alpha-4)X^3+(-\alpha^3+2\alpha)X^2+(\alpha^3+2\alpha^2)X-\alpha^3)\\
=&
5(P)+(U)+(P-U)-7O.\\
\end{align*}
In sum, this gives 
\begin{align*}
(a(Z_1,Y_1))=&2(P)+(U)+(P-U)-(5P)-(2P)-(V)-(2P-V),\\
(b(Z_1,Y_1))=&2(5P)-(2P)-(V)-(2P-V)+O.
\end{align*}

By applying Lemma \ref{lem:ab},
\begin{align}
-(x_1)\diamond (y_1)\sim &5(P)+3(2P)+(U)+(P-U)+3(P+U)+3(2P-U)+(V-U)\nonumber\\
& +(2P-U-V)+(U+V-P)+(U-V+P)-(V)-(2P-V)\nonumber\\
&-3(V+P)+3(V+3P).\label{eq:x1y1}
\end{align}

Now we record other divisors. 
\begin{align*}
(X+\alpha)=&(V-P)+(P-V)-2O,\\
(\alpha X+2Y+\alpha^2)=& (5P)+(U)+(P-U)-3O,\\
(Y)=& 3(2P)-3O.
\end{align*}

The above relations imply
\begin{align*}
\left(\frac{X+\alpha}{Y}\right)=&(V-P)+(P-V)+O-3(2P),\\
\left(\frac{\alpha X+2Y+\alpha^2}{Y} \right)=&(5P)+(U)+(P-U)-3(2P),\\
\end{align*}
and
\begin{align*}
0\sim& \left(-\frac{\alpha(X+\alpha)}{2Y}\right) \diamond\left(\frac{\alpha X+2Y+\alpha^2}{2Y} \right)\\
=&(P)+3(2P)-(U)-(P-U)-3(P+U)-3(2P-U)-(V-U)\\& -(2P-U-V)-(U+V-P)-(U-V+P)+(V)+(2P-V)\\
&+3(V+P)-3(V+3P).\\
\end{align*}

Combining the above equation with \eqref{eq:x1y1} we obtain 
\[(x_1)\diamond (y_1)\sim -6(P)-6(2P).\]

By comparing with equation \eqref{eq:xy}, we conclude,
\[(x_1)\diamond (y_1)\sim (x)\diamond (y).\]

\subsection{The relationship between the cycles}

We consider the integration cycle for the Mahler measure over $P_\alpha$ first. 
It is convenient to make the change of variables $x=x_0^2$ as well as $y_0=y/x_0$. In this case we have
\[(x_0+x_0^{-1})y_0^2+(x_0^2 + (2-\alpha) + x_0^{-2})y_0+(x_0+x_0^{-1})=0.\]
This gives
\begin{align*}
{y_0}_\pm=&\frac{-(x_0^2 + (2-\alpha) + x_0^{-2})}{2(x_0+x_0^{-1})}\\
&\pm \frac{\sqrt{(x_0^2-2x_0+2-\alpha-2x_0^{-1}+x_0^{-2})
(x_0^2+2x_0+2-\alpha+2x_0^{-1}+x_0^{-2})}}{2(x_0+x_0^{-1})}.
\end{align*}
Now write $x_0=e^{i\theta}$ with $0 \leq \theta \leq \pi$. We have
\[{y_0}_\pm=\frac{(\alpha-4 \cos^2\theta)\pm \sqrt{(\alpha -4\cos^2\theta)^2-16 \cos^2\theta
}}{4\cos \theta}.\]

Further taking $t=\cos^2 \theta$, we have that the polynomial inside the square root is 
$16t^2-8(2+\alpha)t+\alpha^2$ which has roots $t=\frac{2+\alpha\pm 2\sqrt{\alpha+1}}{4}$.

When this polynomial takes negative values, both roots are complex conjugate of each other
and both have absolute value 1. We are interested in the cases that the polynomial takes positive values 
and one of the roots has absolute value larger than 1. 

When $0\leq \alpha\leq 4$ this polynomial takes positive values 
for $0\leq t\leq \frac{2+\alpha -2\sqrt{\alpha+1}}{4}$. We can see that in this case $|{y_0}_+|>1$.

When $\alpha \leq -1$, the polynomial inside the square root has no real roots and therefore it is positive for $0\leq t \leq 1$. Both roots are then real. We see that $|{y_0}_-|>1$. 

In order to characterize the homology class given by the integration set, we integrate respect to 
the standard invariant differential $\omega$ of the elliptic curves. Recall that 
\[\omega=\frac{d X}{2Y+(\alpha-2)X+\alpha}.\]
By looking at the transformations, we have
\[d X=-\frac{\alpha(\alpha+1)(dx+dy)}{(x+y-\alpha)^2}.\]
By differentiating $P_\alpha$, we have,
\[(2(y+1)x+y^2+(2-\alpha)y+1)dx+(2(x+1)y+x^2+(2-\alpha)x+1)dy=0.\]
Putting the above together, we obtain,
\begin{align*}
d X=&\frac{\alpha(\alpha+1)(y-x)dx}{(2(x+1)y+x^2+(2-\alpha)x+1)(x+y-\alpha)}\\
=& \frac{\alpha(\alpha+1)y(y-x)dx}{(x+1)(y^2-x)(x+y-\alpha)}.
\end{align*}
Therefore
\begin{align*}
\frac{dX}{2Y+(\alpha-2)X+\alpha}=&\frac{dx}{2(x+1)y+x^2+(2-\alpha)x+1}=\frac{ydx}{(x+1)(y^2-x)}\\
=&\frac{2dx_0}{(2(x_0+x_0^{-1})y_0+x_0^2+(2-\alpha)+x_0^{-2})x_0}.
\end{align*}

At this point, we either have to specify the choice of the root ${y_0}_\pm$ or leave the sign in front of the square-root undetermined. Since all the Mahler measures are non negative, and the integration sets are connected, we can leave the sign to be determined later.  
\[\omega=\pm \frac{2 id \theta}{\sqrt{(\alpha-4\cos^2\theta)^2-16\cos^2 \theta}}.\]

Take $t=\cos^2 \theta$, then $\frac{dt}{\sqrt{t(1-t)}}=-2d\theta$ and
\[\omega=\pm\frac{ id t}{\sqrt{t(1-t)((\alpha-4t)^2-16t)}}\]

In sum, we must consider, for $0\leq \alpha \leq 4$,
\begin{align*}
\int_{\varphi_*(|x|=1)} \omega =&\pm \int_0^\frac{2+\alpha-2\sqrt{\alpha+1}}{4} \frac{2 id t}{\sqrt{t(1-t)((\alpha-4t)^2-16t)}}\\
\end{align*}
and for $\alpha \leq -1$,
\begin{align*}
\int_{\varphi_*(|x|=1)} \omega =&\pm \int_0^1 \frac{2 id t}{\sqrt{t(1-t)((\alpha-4t)^2-16t)}}.
\end{align*}
In both cases, the extra factor $2$ comes from changing $0\leq \theta \leq \pi$ to $0\leq \theta \leq \frac{\pi}{2}$. 

Now we analyze the cycle for $S_\alpha$. 
Make the change of variables $y_0=y_1/x_1^2$. This gives
\[y_0^2+(x_1^2+\alpha x_1+2\alpha +\alpha x_1^{-1}+x_1^{-2})y_0 +1=0\]
and
\begin{align*}
{y_0}_\pm=&\frac{-(x_1^2+\alpha x_1+2\alpha +\alpha x_1^{-1}+x^{-2})}{2}\\
&\pm \frac{\sqrt{(x_1+2+x_1^{-1})(x_1+\alpha-2+x_1^{-1})(x_1^2+\alpha x_1+2(\alpha+1)+\alpha x_1^{-1}+x_1^{-2})}}{2}.\\
\end{align*}
By setting $x_1=e^{i\theta}$ with $0 \leq \theta \leq 2\pi$, we have
\begin{align*}
{y_0}_\pm=&-(2\cos^2 \theta+\alpha \cos \theta+(\alpha-1))\\
&\pm \sqrt{(\cos \theta+1)(2\cos \theta+\alpha-2)(2\cos^2\theta+\alpha\cos \theta +\alpha)}.\\
\end{align*}
Taking $t=\cos \theta$, the polynomial inside the square root
is $(t+1)(2t+\alpha-2)(2t^2+\alpha t+\alpha)$. The roots for the quadratic factor are given by 
$\frac{-\alpha\pm \sqrt{\alpha^2-8\alpha}}{4}$. As before, we are interested in the case when the polynomial inside the square-root takes positive values. 

When $0\leq \alpha \leq 4$, the polynomial is positive for $\frac{2-\alpha}{2}\leq t \leq 1$. 

When
$\alpha \leq -1$, the polynomial is positive for $\frac{-\alpha-\sqrt{\alpha^2-8\alpha}}{4}\leq t \leq 1$. 

In both cases, this leads to a root that has absolute value greater or equal to 1 and another 
that has absolute value less or equal to 1. As observed in the previous case, we
do not have to determine the exact sign of this root as long as each integral is done over a fixed root.

On the other hand, we have,
\[dX=\frac{\alpha(\alpha+1)dZ_1}{4}=\frac{\alpha(\alpha+1)X_1dX_1}{2}=-\frac{\alpha(\alpha+1)(x_1+1)dx_1}{(x_1-1)^3}.\]
We also have
\[2Y+(\alpha-2)X+\alpha=\frac{\alpha(\alpha+1)Y_1}{4}=\frac{\alpha(\alpha+1)(y_1^2-x_1^4)}{y_1(x_1-1)^3(x_1+1)}.\]
Therefore,
\begin{align*}
\frac{dX}{2Y+(\alpha-2)X+\alpha}=&-\frac{(x_1+1)^2y_1dx_1}{(y_1^2-x_1^4)}=\frac{(x_1+1)^2y_1dx_1}{(x_1^4+\alpha x_1^3+2\alpha x_1^2+\alpha x_1 +1)y_1+2x_1^4}.
\end{align*}
Let $y_{0}=y_1/x_1^2$. Then
\begin{align*}
\frac{dX}{2Y+(\alpha-2)X+\alpha}=&\frac{(x_1+2+x_1^{-1})y_0dx_1}{((x_1^2+\alpha x_1+2\alpha +\alpha x_1^{-1} +x_1^{-2})y_0+2)x_1}.\\
\end{align*}
Writing $x_1=e^{i\theta}$ with $0\leq \theta\leq 2\pi$, this leads to 
\[\omega= \pm \frac{(1+\cos \theta)i d  \theta }{\sqrt{(\cos \theta+1)(2\cos \theta+\alpha-2)(2\cos^2\theta+\alpha\cos \theta +\alpha)}}.\]

Take $t=\cos \theta$. Then $-\frac{dt}{\sqrt{1-t^2}}=d \theta$ and 
\[\omega =\pm \frac{ i d t }{\sqrt{(1-t)(2 t+\alpha-2)(2 t^2+\alpha t +\alpha)}}.\]

In sum, for $0\leq \alpha \leq 4$, we must consider
\[\int_{\varphi_*(|x_1|=1)} \omega =\pm \int_\frac{2-\alpha}{2}^1 \frac{2 i d t }{\sqrt{(1-t)(2 t+\alpha-2)(2 t^2+\alpha t +\alpha)}}\]
and for $\alpha\leq -1$, 
\[\int_{\varphi_*(|x_1|=1)} \omega = \pm \int_\frac{-\alpha -\sqrt{\alpha^2-8\alpha}}{4}^1 \frac{2i d t }{\sqrt{(1-t)(2 t+\alpha-2)(2 t^2+\alpha t +\alpha)}}.\]

The result is completed with the following statement which covers the necessary identities, except for the boundary cases, which can be deduced by continuity.  

\begin{lem} 
 For $0< \alpha <8 $, we have 
\begin{equation}\label{eq:0a4}
2\int_0^\frac{2+\alpha-2\sqrt{\alpha+1}}{4} \frac{d t}{\sqrt{t(1-t)((\alpha-4t)^2-16t)}}=\int_\frac{2-\alpha}{2}^1 \frac{dt }{\sqrt{(1-t)(2 t+\alpha-2)(2 t^2+\alpha t +\alpha)}}.
\end{equation}

For $\alpha < -1$, we have
\begin{equation}\label{eq:a-1}
\int_0^1 \frac{d t}{\sqrt{t(1-t)((\alpha-4t)^2-16t)}}=\int_\frac{-\alpha -\sqrt{\alpha^2-8\alpha}}{4}^1 \frac{d t }{\sqrt{(1-t)(2 t+\alpha-2)(2 t^2+\alpha t +\alpha)}}.
\end{equation}

\end{lem}
\begin{proof}

First consider the change of variables  $t=\frac{\alpha s -\alpha+2}{2}$. Then for $0< \alpha$,
\[\int_\frac{2-\alpha}{2}^1 \frac{dt }{\sqrt{(1-t)(2 t+\alpha-2)(2 t^2+\alpha t +\alpha)}}
 =\int_0^1 \frac{ds}{\sqrt{s(1-s)(\alpha^2 s^2+\alpha(4-\alpha) s +4)}},\]
and for $\alpha < -1$,
\begin{align*}
&\int_\frac{-\alpha -\sqrt{\alpha^2-8\alpha}}{4}^1 \frac{dt }{\sqrt{(1-t)(2 t+\alpha-2)(2 t^2+\alpha t +\alpha)}}\\
 =&-\int_\frac{\alpha -4 -\sqrt{\alpha^2-8\alpha}}{2\alpha}^1 \frac{ds}{\sqrt{s(1-s)(\alpha^2 s^2+\alpha(4-\alpha) s +4)}}.\\
\end{align*}
The right-hand sides of the above equations are related to the formulas found by Rogers and Zudilin \cite{RZ12} and Bertin and Zudilin \cite{BZ}, and this allows us to use their changes of variables to manipulate those sides of the equations.
However, the left-hand sides do not appear in those works. As we will eventually see, the connection between the two sides are given by periods in two isogenous elliptic curves.

Notice that we can modify the integration limits in the last integral by the involution $s=\frac{1-w}{1+\alpha w}$, which  gives  for $\alpha < -1$,
\begin{align*}
&-\int_\frac{\alpha -4 -\sqrt{\alpha^2-8\alpha}}{2\alpha}^1 \frac{ds}{\sqrt{s(1-s)(\alpha^2 s^2+\alpha(4-\alpha) s +4)}}\\
 =&\int_0^\frac{\alpha -4 +\sqrt{\alpha^2-8\alpha}}{2\alpha} \frac{dw}{\sqrt{w(1-w)(\alpha^2 w^2+\alpha(4-\alpha) w +4)}}.
\end{align*}
In sum, we have to prove,
for $0< \alpha < 8$. 
\begin{equation} \label{arghh2}
2\int_0^\frac{2+\alpha-2\sqrt{\alpha+1}}{4} \frac{d t}{\sqrt{t(1-t)((\alpha-4t)^2-16t)}}
 =\int_0^1 \frac{ds}{\sqrt{s(1-s)(\alpha^2 s^2+\alpha(4-\alpha) s +4)}}
 \end{equation}
and for $\alpha < -1$,
\begin{equation}\label{arghh}
\int_0^1 \frac{d t}{\sqrt{t(1-t)((\alpha-4t)^2-16t)}}=\int_0^\frac{\alpha -4 +\sqrt{\alpha^2-8\alpha}}{2\alpha} \frac{dw}{\sqrt{w(1-w)(\alpha^2 w^2+\alpha(4-\alpha) w +4)}}.
\end{equation}

First we concentrate on equation \eqref{arghh}. Consider the change $t=\frac{1}{1+\frac{4u}{\alpha^2}}$. Then the left-hand side of equation \eqref{arghh} becomes 
\begin{equation}\label{eq:first}
\int_0^1 \frac{d t}{\sqrt{t(1-t)((\alpha-4t)^2-16t)}}=\frac{1}{2}\int_0^\infty \frac{du}{
\sqrt{u\left(u^2+2\left(\frac{\alpha^2}{4} - \alpha - 2\right)u+\frac{\alpha^3}{16}(\alpha-8)\right)
}}.
\end{equation}

Consider $w=\frac{1}{1+v}$. The right-hand side of equation \eqref{arghh} becomes 
\begin{align}
&\int_0^\frac{\alpha -4 +\sqrt{\alpha^2-8\alpha}}{2\alpha} \frac{dw}{\sqrt{w(1-w)(\alpha^2 w^2+\alpha(4-\alpha) w +4)}}\nonumber \\
\label{eq:second}=&\frac{1}{2} \int_\frac{\alpha^2-4\alpha-8-\alpha \sqrt{\alpha^2-8\alpha}}{8}^\infty \frac{dv}{\sqrt{v(v^2-\left(\frac{\alpha^2}{4}-\alpha-2\right)v+\alpha+1)}}.
\end{align}

The integrals on the right-hand sides of equations \eqref{eq:first} and \eqref{eq:second} correspond to the same
periods in isogenous elliptic curves. We can use the standard isogeny of degree 2 for the Weierstrass form $y^2=x(x^2+ax+b)$ to describe the change of variables between them. More precisely,
$u=v-\left(\frac{\alpha^2}{4}-\alpha-2\right)+\frac{\alpha+1}{v}$ yields
\begin{align*}
&\int_0^\infty \frac{du}{\sqrt{u\left(u^2+2\left(\frac{\alpha^2}{4} - \alpha - 2\right)u+\frac{\alpha^3}{16}(\alpha-8)\right)}}\\
=&\int_\frac{\alpha^2-4\alpha-8-\alpha \sqrt{\alpha^2-8\alpha}}{8}^\infty \frac{dv}{\sqrt{v\left(v^2 -\left(\frac{\alpha^2}{4}-\alpha-2\right)v+\alpha+1\right)}}.
\end{align*}
This concludes the proof of equation \eqref{arghh} and therefore of equation \eqref{eq:a-1}. 

For \eqref{arghh2}, consider the following observation. If we set $\beta=-\frac{8}{\alpha}$, then 
 \[\int \frac{dt}{\sqrt{t(1-t)((\alpha-4t)^2-16t)}}=\frac{|\beta|}{4}\int \frac{dt}{\sqrt{t(1-t)(\beta^2 t^2 + \beta(4-\beta)t + 4)}}.\]

 Applying the above transformation to equation \eqref{arghh}, we get, for $0< \beta < 8$,
\begin{align*}
& \frac{\beta}{4}\int_0^1 \frac{dt}{\sqrt{t(1-t)(\beta^2 t^2 + \beta(4-\beta)t + 4)}}\\
 =&
\frac{4}{\left|\frac{-8}{\beta}\right|} \int_0^\frac{2+\beta -2\sqrt{\beta +1}}{4}
\frac{d w}{\sqrt{w(1-w)((\beta-4w)^2-16w)}},
\end{align*}
which implies equations \eqref{arghh2} and \eqref{eq:0a4}.

\end{proof}

\section{The families from Theorem \ref{thm:BZ2}} \label{sec:2}

\subsection{The relationship between the regulators}

In this section we work with the families 
\[Q_\alpha(x_2,y_2)=(x_2^2+x_2+1)y_2^2+\alpha x_2(x_2+1) y_2 + x_2(x_2^2+x_2+1)\]
and
\[R_\beta(x_3,y_3)=(x_3^2+x_3+1)y_3^2+(x_3^4+\beta x_3^3+(2\beta-4)x_3^2+\beta x_3+1)y_3+x_3^2(x_3^2+x_3+1).\]

For $Q_\alpha(x_2,y_2)$, Boyd \cite{Bo98} writes 
\[Y_2^2=h_2(X_2^2),\]
where 
\[h_2(Z_2)= (\alpha^2 -9)Z_2^3 - (2\alpha^2 - 3)Z_2^2 +(\alpha^2 + 5)Z_2 + 1,\]
and
\begin{align*}
X_2(x_2,y_2)= & \frac{x_2+1}{x_2-1},& & x_2(X_2,Y_2)=\frac{X_2+1}{X_2-1},\\
Y_2(x_2,y_2)=& \frac{4(2(x_2^2+x_2+1)y_2+\alpha x_2(x_2+1))}{(x_2-1)^3},& & y_2(X_2,Y_2)=\frac{Y_2-\alpha X_2 (X_2^2-1)}{(X_2-1)(3X_2^2+1)}.\\ 
\end{align*}

By applying the transformation 
\[Z=(\alpha^2-9)(Z_2-1), \qquad W=(\alpha^2-9)Y_2,\]
we obtain 
\[F_\alpha: W^2=Z^3 +(\alpha^2-24)Z^2-16(\alpha^2 -9)Z.\]

Our goal as before is to compute  $(x_2)\diamond (y_2)$ in $F_\alpha$. We will do this by applying 
Lemma \ref{lem:ab}. Thus we take
\begin{align*}
a(Z_2,Y_2)=&\frac{Y_2-\alpha(Z_2-1)}{Y_2+\alpha(Z_2-1)}=\frac{W-\alpha Z}{W+\alpha Z},\\
b(Z_2,Y_2)=&-\frac{2(3Z_2+1)}{Y_2+\alpha(Z_2-1)}=-\frac{2(3Z+4(\alpha^2-9))}{W+\alpha Z},\\
\end{align*}
and we can easily see that 
\[a(X_2^2,Y_2)x_2(X_2,Y_2)+b(X_2^2,Y_2) y_2(X_2,Y_2)=1.\]
Lemma \ref{lem:ab} still applies with the same change $x_2\rightarrow \frac{1}{x_2}$ and $y_2\rightarrow \frac{1}{y_2}$ that leads to $X_2\rightarrow -X_2$ and $Y_2\rightarrow Y_2$. 

We consider the following points of $F_\alpha$ (in $Z, W$ coordinates).
\begin{align*}
P=&(0,0),\\
S_\pm=&(\pm 4\alpha+12,\alpha(\pm 4\alpha+12)),\\
T=&\left(-\frac{4(\alpha^2-9)}{3},\frac{4i(\alpha-3)\alpha(\alpha+3)}{3\sqrt{3}}\right),\\
\end{align*}
where $S_++S_-=P$. Thus we rename $S$ to be $S_+$ and $P-S$ to be $S_-$. Notice also that $2P=O$

We compute some divisors. 
\begin{align*}
(W-\alpha Z)=& (S)+(P-S)+(P)-3O,\\ 
(W+\alpha Z)=& (-S)+(P+S)+(P)-3O,\\ 
(3Z+4(\alpha^2-9))=& (T)+(-T)-2O.
\end{align*}
This leads to 
\begin{align*}
(a(Z_2,Y_2))=&(S)+(P-S)-(-S)-(P+S),\\
(b(Z_2,Y_2))=&(T)+(-T)+O-(-S)-(P+S)-(P),
\end{align*}
and 
\[(a(Z_2,Y_2))\diamond (b(Z_2,Y_2))=2(S-T)+2(S+T)-2(P+S+T)-2(P+S-T)+4(S)-4(P+S).\]

Finally, by Lemma \ref{lem:ab}, we conclude,
\begin{equation}\label{eq:x2y2}
-(x_2)\diamond (y_2)\sim 2(S-T)+2(S+T)-2(P+S+T)-2(P+S-T)+4(S)-4(P+S).
\end{equation}

We now consider the case of $R_\beta(x_3,y_3)$. We have
\[Y_3^2=h_3(X_3^2),\]
where
\[h_3(Z_3)=(\beta^2-\beta-2)Z_3^3+(-2\beta^2+11\beta-2)Z_3^2+(\beta^2-11\beta+26)Z_3+\beta-6,\]
and
\begin{align*}
X_3(x_3,y_3)=&\frac{x_3+1}{x_3-1},\\
Y_3(x_3,y_3)=&\frac{4(2(x_3^2+x_3+1)y_3+x_3^4+\beta x_3^3+(2\beta-4)x_3^2+\beta x_3+1)}
{(x_3-1)^3(x_3+1)},\\
x_3(X_3,Y_3)=&\frac{X_3+1}{X_3-1},\\
y_3(X_3,Y_3)=&\frac{2X_3Y_3-(2\beta-1)X_3^4+(2\beta-10)X_3^2+1}{(X_3-1)^2(3X_3^2+1)}.\\
\end{align*}
By applying the transformation 
\[Z=(\beta^2-\beta-2)Z_3-(\beta^2-5\beta-6), \quad W=(\beta^2-\beta-2)Y_3,\]
we obtain
\[W^2=Z^3 + (\beta^2-4\beta-20)Z^2 - 16(\beta^2-4\beta -5)Z.\]
Notice that this is precisely $F_{\beta-2}$. 

We proceed to compute the diamond operation $(x_3)\diamond (y_3)$.
Using the usual strategy of Lemma \ref{lem:ab}, we find 
\begin{align*}
a(Z_3,W_3)=&\frac{(Z_3+1)Y_3-(2\beta-1)Z_3^2+(2\beta-10)Z_3+1}{(Z_3-1)Y_3}\\
=&\frac{ZW+2(\beta^2-3\beta-4)W-(2\beta-1)Z^2-2(\beta^3-5\beta^2-10\beta-4)Z+
16(\beta^3-3\beta^2-9\beta-5)}{W(Z-4(\beta+1))}\\
b(Z_3,W_3)=&-\frac{3Z_3+1}{Y_3}=-\frac{3Z+4(\beta^2-4\beta-5)}{W}\\
\end{align*}
and one can easily see that 
\[a(X_3^2,Y_3)x_3(X_3,Y_3)+b(X_3^2,Y_3) y_3(X_3,Y_3)=1.\]
Lemma \ref{lem:ab} still applies with the same change $x_3\rightarrow \frac{1}{x_3}$ 
and $y_3\rightarrow \frac{1}{y_3}$ that leads to $X_3\rightarrow -X_3$ and 
$Y_3\rightarrow Y_3$.


We consider the following points (in $Z,W$ coordinates), 
\begin{align*}
P=&(0,0),\\
A=&\left(\frac{-(\beta^2-4\beta-20)+\sqrt{\beta^4 - 8\beta^3 + 40\beta^2 - 96\beta + 80}}{2},0\right),\\
A+P=&\left(\frac{-(\beta^2-4\beta-20)-\sqrt{\beta^4 - 8\beta^3 + 40\beta^2 - 96\beta + 80}}{2},0\right),\\
S=&(4(\beta+1),4(\beta-2)(\beta+1)),\\
2S=&(16,-16),\\
T=&\left(-\frac{4(\beta-5)(\beta+1)}{3},\frac{4i(\beta-5)(\beta-2)(\beta+1)}{3\sqrt{3}}\right),\\
P-S=&(4(5-\beta),-4(\beta-5)(\beta-2)),\\
P-2S=&((5-\beta)(\beta+1),(\beta-5)(\beta+1)).\\
\end{align*}
Notice that the points $P,S,T$ are the same that were previously considered in $F_\alpha$. The formulas are different since they depend on the parameter $\beta$.

We then obtain 
\begin{align*}
(W)=&(P)+(A)+(A+P)-3O,\\
(Z-4(\beta+1)) = &(S)+(-S)-2O,\\
(3Z+4(\beta-5)(\beta+1))=&(T)+(-T)-2O,\\
\end{align*}
and
\begin{align*}
&(ZW+2(\beta^2-3\beta-4)W-(2\beta-1)Z^2-2(\beta^3-5\beta^2-10\beta-4)Z+16(\beta^3-3\beta^2-9\beta-5))\\
&=3(S)+(P-S)+(P-2S)-5O.\\
\end{align*}

This implies
\begin{align*}
(a(Z_3,Y_3))=&2(S)+(P-S)+(P-2S)-(P)-(A)-(A+P)-(-S),\\
(b(Z_3,Y_3))=&(T)+(-T)+O-(P)-(A)-(A+P).\\
\end{align*}

Thus,
\begin{align}\label{eq:x3y3}
-(x_3)\diamond (y_3)=&(a(X_3^2,Y_3))\diamond (b(X_3^2,Y_3))\nonumber\\
=&3(S-T)+3(S+T)+4(S)-4(P+S)-(P+S+T)-(P+S-T)\nonumber\\
&+(2S)-(P+2S)-(P+2S+T)+(P-2S+T)-2(S+A)\nonumber \\&+(2S+A)-2(S+A+P)+(2S+A+P).
\end{align}


Now consider
\begin{align*}
(W-3Z-4(\beta-5)(\beta+1))=&(S)+(P+S)+(P-2S)-3O\\
\end{align*}
and
\begin{align*}
\left(\frac{W-3Z-4(\beta-5)(\beta+1)}{W}\right) = & (S)+(P+S)+(P-2S)-(P)-(A)-(A+P)\\
\left(\frac{3Z+4(\beta-5)(\beta+1)}{W}\right) = & (T)+(-T)+O-(P)-(A)-(A+P).\\
\end{align*}
Combining the above divisors, we have
\begin{align*}
&\left(\frac{W-3Z-4(\beta-5)(\beta+1)}{W}\right) \diamond \left(\frac{3Z+4(\beta-5)(\beta+1)}{W}\right)=
(S-T)+(S+T)\\&+(P+S+T)+(P+S-T)+(P-2S+T)+(P-2S-T)
+(2S)-(P+2S)\\&-2(S+A)-2(S+A+P)+(2S+A)+(2S+A+P).
\end{align*}

By comparing with equations \eqref{eq:x2y2} and \eqref{eq:x3y3}, we get 
\[(x_2)\diamond (y_2) \sim (x_3)\diamond (y_3).\]

\subsection{The relationship between the cycles}

We start by considering $Q_\alpha$.
It is convenient to make the change of variables $x_2=x_0^2$ and $y_2=y_0x_0$. We then consider
\[(x_0^2+1+x_0^{-2})y_0^2+\alpha(x_0+x_0^{-1}) y_0 + (x_0^2+1+x_0^{-2})=0.\]
In this case we have
\begin{align*}
{y_0}_\pm = &\frac{-\alpha (x_0+x_0^{-1})}{2(x_0^2+1+x_0^{-2})}\\&\pm \frac{\sqrt{-(2x_0^2-\alpha x_0+2-\alpha x_0^{-1}+2x_0^{-2})(2x_0^2+\alpha x_0+2+\alpha x_0^{-1}+2x_0^{-2})}}{2(x_0^2+1+x_0^{-2})}.
\end{align*}
Write $x_0=e^{i\theta}$ with $0\leq \theta \leq \pi$. We have
\[{y_0}_\pm = \frac{-\alpha \cos \theta \pm \sqrt{ \alpha^2 \cos^2 \theta - (4\cos^2 \theta-1)^2}}{4\cos^2 \theta -1}.\]
If we take $t=\cos^2 \theta$, the polynomial inside the square root is $-16t^2+(8+\alpha^2)t-1$ and is positive when $\alpha\geq 4$ for $\frac{8+\alpha^2-\alpha\sqrt{\alpha^2+16}}{32}\leq t \leq 1$. As observed in the previous section, we do not have to determine which of the roots has absolute value greater or equal than 1. 

We evaluate $\omega=\frac{d Z}{2W}$. First we have
\[dZ=-4(\alpha^2-9)\frac{x_2+1}{(x_2-1)^3}dx_2.\]
Therefore,
\begin{align*}
\frac{d Z}{2W}=&-\frac{x_2+1}{2(2(x_2^2+x_2+1)y_2+\alpha x_2(x_2+1))}dx_2\\
=&-\frac{x_0+x_0^{-1}}{2(x_0^2+1+x_0^{-2})y_0+\alpha (x_0+x_0^{-1})}\frac{dx_0}{x_0}.\\
\end{align*}
Writing $x_0=e^{i\theta}$, this leads to 
\[\omega =\pm\frac{\cos \theta i d \theta}{\sqrt{ \alpha^2 \cos^2 \theta - (4\cos^2 \theta-1)^2}}.\]
Take $t=\cos^2 \theta$ and $d \theta= -\frac{dt}{2\sqrt{t(1-t)}}$. 
\[\omega=\pm\frac{i d t}{2\sqrt{(1-t)(\alpha^2t -(4t-1)^2)}}\]
We must consider for $\alpha\geq 4$, 
\[\int_{\varphi_*(|x_2|=1)} \omega =\pm \int_{\frac{8+\alpha^2-\alpha\sqrt{\alpha^2+16}}{32}}^1 \frac{i d t}{\sqrt{(1-t)(\alpha^2t -(4t-1)^2)}},\]
where the extra factor $2$ comes from changing $0\leq \theta \leq \pi$ to $0\leq \theta \leq \frac{\pi}{2}$.

We consider $R_\beta$. 
We write  $y_3=y_0x_3$. Then we have
\[(x_3+1+x_3^{-1})y_0^2+(x_3^2+\beta x_3+(2\beta-4)+\beta x_3^{-1}+x_3^{-2})y_0+(x_3+1+x_3^{-1})=0.\]
\begin{align*}
{y_0}_\pm =&\frac{-(x_3^2+\beta x_3+(2\beta-4)+\beta x_3^{-1}+x_3^{-2})}{2(x_3+1+x_3^{-1})}\\
&\pm \frac{\sqrt{(x_3+2+x_3^{-1})
(x_3+(\beta-4)+x_3^{-1})(x_3^2+(\beta+2) x_3+2(\beta-1)+(\beta+2)x_3^{-1}
+x_3^{-2})}}
 {2(x_3+1+x_3^{-1})}.
 \end{align*}
Write $x_3=e^{i\theta}$ with $0\leq \theta \leq 2\pi$,
\begin{align*}
{y_0}_\pm =&\frac{-(4\cos^2 \theta+2\beta \cos \theta + (2\beta-6))}{2(1+2 \cos \theta)}\\
&\pm \frac{\sqrt{2(1+\cos \theta)
((\beta-4)+2\cos \theta)(4\cos^2 \theta +2(\beta+2)\cos \theta +2(\beta-2))}}
 {2(1+2 \cos \theta)}.
 \end{align*}
Write $t=\cos \theta$, the polynomial inside the square root is 
\[2(1+t)((\beta-4)+2t)(4t^2 +2(\beta+2)t+2(\beta-2)).\]
For $\beta \geq 6$, we must consider $\frac{-(\beta+2)+\sqrt{\beta^2-4\beta+20}}{4}\leq t\leq 1$ for the polynomial inside the square root to be positive. 
 
We evaluate $\omega=\frac{dZ}{2W}$. We have 
\[dZ=-4(\beta^2-\beta-2)\frac{x_3+1}{(x_3-1)^3} dx_3.\]
\[\frac{dZ}{2W}= - \frac{x_3+2+x_3^{-1}}{2(2(x_3+1+x_3^{-1})y_0+x_3^2+\beta x_3+(2\beta-4)+\beta x_3^{-1}+x_3^{-2})}\frac{dx_3}{x_3}\]

\[\omega = \pm \frac{(1+\cos \theta)i d  \theta}{\sqrt{2(1+\cos \theta)
((\beta-4)+2\cos \theta)(4\cos^2 \theta +2(\beta+2)\cos \theta +2(\beta-2))}}\]
Take $t=\cos \theta$, then 
\[\omega = \pm \frac{i d  t}{2\sqrt{(1-t)
((\beta-4)+2t)(2t^2 +(\beta+2)t +(\beta-2))}}.\]

Therefore, we must consider
\begin{align*}
&\int_{\varphi_*(|x_3|=1)} \omega \\=&\pm \int_{\frac{-(\beta+2)+\sqrt{\beta^2-4\beta+20}}{4}}^1
\frac{i d  t}{\sqrt{(1-t)
((\beta-4)+2t)(2t^2 +(\beta+2)t +(\beta-2))}},
\end{align*}
where the extra factor $2$ comes from changing $0\leq \theta \leq 2\pi$ to $0\leq \theta \leq \pi$.

Since we have $\alpha=\beta-2$, we must prove for $\alpha \geq 4$
\begin{align*}
&\int_{\frac{8+\alpha^2-\alpha\sqrt{\alpha^2+16}}{32}}^1 \frac{dt}{\sqrt{(1-t)(\alpha^2t -(4t-1)^2)}}\\
=&\int_{\frac{-(\alpha+4)+\sqrt{\alpha^2+16}}{4}}^1
\frac{d s}{\sqrt{(1-s)
(2s+\alpha-2)(2s^2 +(\alpha+4)s +\alpha)}}.
\end{align*}
In fact, we can go from one side to the other by setting
\[t=\frac{(\alpha+1)s+\alpha-1}{2(2s+\alpha-2)}.\] 
(This change of variables can be deduced from \cite{BZ2}.)

\section{Conclusion}
We have seen a way to reinterpret the identities between the Mahler measures of genus 2 curves 
by using the regulator. The use of the regulator highlights the expected relationship of each formula with the $L$-value. 
As a method of proof, this is limited by the requirement that the parameter leads to an integral model. 

All the polynomials involved in our examples can be seen as reciprocal. A consequence of this is that the product of both roots $y_\pm$ 
is 1, which allows us to always pick exactly one root for the integration. The reciprocity of the polynomials 
is also closely related to the property that the involution $x\rightarrow \frac{1}{x}$, $y\rightarrow \frac{1}{y}$
is equivalent to 
$X\rightarrow -X$, $Y\rightarrow Y$. This condition is key for the application 
of Lemma \ref{lem:ab}. A question remains if the method can be extended to other examples where this is 
not satisfied.

It would be interesting to see if this method can be applied to more complicated identities of higher genus, such as those recently found numerically by 
Liu and Qin \cite{HH}.

\section*{Acknowledgments}

The authors are grateful to Hang Liu for making them aware of his work with Hourong Qin and 
to the referee for the very careful reading and the thoughtful corrections.

\bibliographystyle{amsalpha}


\providecommand{\bysame}{\leavevmode\hbox to3em{\hrulefill}\thinspace}
\providecommand{\MR}{\relax\ifhmode\unskip\space\fi MR }
\providecommand{\MRhref}[2]{%
  \href{http://www.ams.org/mathscinet-getitem?mr=#1}{#2}
}
\providecommand{\href}[2]{#2}

\end{document}